\documentclass[11pt]{amsart}
\usepackage{amsfonts}
\usepackage{amsmath,amssymb}
\usepackage{mathrsfs}
\usepackage{amssymb, amsmath, mathrsfs,graphicx}
\usepackage{subfigure,wrapfig}

\newtheorem{theo}{Theorem}[section]

\newtheorem{lem}{Lemma}[section]

\newcommand{\eop}{\hspace{\fill} $\;\;\;  \Box$}

\font\uj=cmssbx10 scaled \magstep1
\DeclareMathDelimiter{\Norm}{\mathord}{largesymbols}{"3E}{largesymbols}{"3E}

\begin{document}
\title[Hyperbolicity of minimal periodic orbits]{Hyperbolicity of minimal periodic orbits}
\author[C.-Q. Cheng \& M. Zhou]{Chong-Qing CHENG and Min Zhou}
\address{Department of Mathematics\\ Nanjing University\\ Nanjing 210093,
China\\ }
\email{chengcq@nju.edu.cn, minzhou@nju.edu.cn}


\begin{abstract}  For positive definite Lagrange systems with two degrees of freedom, it is a typical phenomenon that all minimal periodic orbits are hyperbolic.
\end{abstract}
\maketitle

\section{Introduction}
The configuration manifold considered here is a two-dimensional torus $M=\mathbb{T}^2$. We assume that $L$: $TM\to\mathbb{R}$ is an autonomous $C^r$-Lagrangian with $r\ge 4$, strictly convex on each tangent fiber, namely, the Hessian matric of $L$ in $\dot x$ is positive definite for each $x\in\mathbb{T}^2$. Therefore, it makes sense to study the minimal measures by applying the Mather theory for Tonelli Lagrangian \cite{Mat}.

Let $\mathfrak{M}(L)$ be the set of Borel probability measures on $TM$ which are invariant for $\phi_L^t$, the Lagrange flow determined by $L$. For each $\mu\in\mathfrak{M}(L)$ the rotation vector $\rho(\mu)\in H_1(\mathbb{T}^2,\mathbb{R})$ is defined such that for every closed 1-form $\eta$ on $\mathbb{T}^2$ one has
$$
\langle[\eta],\rho(\mu)\rangle=\int\eta d\mu.
$$
Let $\mathfrak{M}_{\omega}(L)=\{\mu\in\mathfrak{M}(L):\rho(\mu)=\omega\}$, an invariant measure $\mu$ is called minimal with the rotation vector $\omega$ if
$$
\int Ld\mu=\inf_{\nu\in\mathfrak{M}_{\omega}(L)}\int Ld\nu.
$$
A rotation vector $\omega\in H_1(\mathbb{T}^2,\mathbb{R})$ is called resonant if there exists a non-zero integer vector $k\in\mathbb{Z}^2$ such that $\langle\omega,k\rangle=0$. For two-dimensional space, it uniquely determines an irreducible element $g\in H_1(\mathbb{T}^2,\mathbb{Z})$ and a positive number $\lambda>0$ such that $\omega=\lambda g$ if the rotation vector $\omega$ is resonant. Let $\alpha$ and $\beta$ denote the $\alpha$- and $\beta$-function of the Lagrangian $L$ respectively. The Fenchel-Legendre transformation $\mathscr{L}_{\beta}$: $H_1(\mathbb{T}^2,\mathbb{R}) \to H^1(\mathbb{T}^2,\mathbb{R})$ is defined as follows
$$
c\in\mathscr{L}_{\beta}(\omega)\ \ \iff \ \ \alpha(c)+\beta(\omega)=\langle c,\omega\rangle.
$$

Since the configuration manifold is two-dimensional, each orbit $(\gamma,\dot\gamma)$ in the support of minimal measure $\mu$ is  periodic if $\rho(\mu)$ is resonant. Indeed, if $\rho(\mu)=\lambda g$ where $g$ is irreducible and $\lambda>0$, all periodic orbits in the support share the same homology class. It is guaranteed by the Lipschitz property of Aubry set discovered by Mather.

In this paper, a periodic orbit is called {\it minimal} if there exists some cohomology class $c\in H^1(\mathbb{T}^2,\mathbb{R})$ such that the periodic orbit is located in the support of the $c$-minimal measure. The main result of this paper is the following

\begin{theo}\label{mainth}
There exists a residual set $\mathfrak{P}\subset C^r(\mathbb{T}^2,\mathbb{R})$ with $r\ge4$ such that for each $P\in\mathfrak{P}$ and all $c\in\{\mathscr{L}_{\beta}(\lambda g):\forall\,\lambda\ge 0,\forall\,g\in H_1(\mathbb{T}^2,\mathbb{Z})\}$, all minimal periodic orbits of the flow determined by $L+P-\langle c,\dot x\rangle$ are hyperbolic.
\end{theo}

\noindent{\bf Remark}. The genericity proved here is in the sense of Ma\~n\'e, the perturbation is applied only to potential. It is stronger than the usual genericity, where the perturbations depend both on $x$ and on $\dot x$.

Let us recall general definition of $c$-minimal curve. Let $\eta_c(x)$ denote a closed 1-form $\langle\eta_c(x),dx\rangle$ evaluated at $x$, with its first co-homology class $[\langle\eta_c(x),dx\rangle]=c\in H^1(\mathbb{T}^2,\mathbb{R})$. We introduce a Lagrange multiplier $\eta_c=\langle\eta_c(x),\dot x\rangle$ and call it closed 1-form also.  For autonomous system, a curve $\gamma$: $\mathbb{R}\to M$ is called $c$-minimal if for any $t<t'$ and any absolutely continuous curve $\zeta$: $[t,t']\to M$ such that $\zeta(t)=\gamma(t)$, $\zeta(t')=\gamma(t')$  one has
$$
\int_t^{t'}(L-\eta_c)(\gamma(t),\dot\gamma(t))dt\le\int_t^{t'}(L-\eta_c)(\zeta(t),\dot\zeta(t))dt.
$$
As the configuration manifold is two-dimensional torus, each {\it minimal} curve must be periodic if it has resonant rotation vector.

\section{Variational Set-up}
\setcounter{equation}{0}

The result is proved by variational method. Let us fix an irreducible class $g\in H_1(\mathbb{T}^2,\mathbb{Z})$ and show the hyperbolicity of minimal periodic orbits with this homological class. By adding a constant and a closed $1$-form to $L$, we assume $\alpha(0)=\min\alpha=0$.

Given an irreducible  element $g\in H_1(\mathbb{T}^2,\mathbb{Z})$, we have a channel
$$
\mathbb{C}_g=\bigcup_{\lambda>0}\mathscr{L}_{\beta}(\lambda g)\ \ \ \text{\rm or}\ \ \ \mathbb{C}_g=\bigcup_{\lambda\ge\lambda_0>0}\mathscr{L}_{\beta}(\lambda g).
$$
In former case, $\alpha(c)>\min\alpha$ for each $c\in\mathscr{L}_{\beta}(\lambda g)$ and each $\lambda>0$. In latter case, $\alpha(c)=\min\alpha$ for $c\in\mathscr{L}_{\beta}(\lambda_0g)$.
In general, $\mathbb{C}_g$ has a foliation of lines $I_{g,E}=\mathbb{C}_g\cap H^{-1}(E)$.  Indeed, as the configuration manifold is $\mathbb{T}^2$, each ergodic minimal measure is supported on a periodic orbit if the rotation vector is resonant. Let $\gamma$ be a periodic curve so that $(\gamma,\dot\gamma)$ is a periodic orbit located in the Mather set for the class $c\in I_{g,E}$, the group $H_1(\mathbb{T}^2,\gamma,\mathbb{Z})$ is an one-dimensional lattice, two generators are denoted by $g_0$ and $g'_0=-g_0$. For a homology class $g'\in H_1(\mathbb{T}^2,\gamma,\mathbb{Z})$ we consider the quantity
$$
h(g',c)=\lim_{T\to\infty}\inf_{\stackrel {\zeta(0)=\zeta(T)\in\gamma}{\scriptscriptstyle [\zeta]=g'}} \int_0^T(L-\eta_c)(\zeta(t),\dot\zeta(t))dt+T\alpha(c).
$$
The identity $h(g_0,c)=h(g'_0,c)$ holds for all $c\in I_{g,E}$ if and only if the set $I_{g,E}$ is a singleton.  In this case, there exists an invariant two-dimensional torus foliated into a family of periodic orbits with the same homological class $g$, each of them is $c$-minimal. By the result of \cite{BC}, it is generic that all minimal measures consist of at most $n+1$-ergodic components. Therefore, the set $I_{g,E}$ is a segment of line in general case. At one endpoint of the interval $I_{g,E}$ one has $h(g_0,c)=0$ and $h(g'_0,c)>0$ and at another endpoint one has $h(g'_0,c)=0$ and $h(g_0,c)>0$. Since all cohomology classes on $I_{g,E}$ share the same Mather set, it makes sense to write $\tilde{\mathcal{M}}(c)=\tilde{\mathcal{M}}(E,g)$ with $E=\alpha(c)$ and $c\in\mathbb{C}_g$.

\begin{theo}\label{theo1}
Given a class $g\in H_1(\mathbb{T}^2, \mathbb{Z})$ and a closed interval $[E_a,E_d]\subset \mathbb{R}_+$ with $E_a>\min\alpha$, there exists an open-dense set $\mathfrak{O}\subset C^{r}(\mathbb{T}^2,\mathbb{R})$ with $r\ge 4$ such that for each $P\in\mathfrak{O}$, each $E\in[E_a,E_d]$, the Mather set $\tilde{\mathcal{M}}(E,g)$ for $L+P$ consists of hyperbolic periodic orbits. Indeed, except for finitely many $E_j\in[E_a,E_d]$ where the Mather set consists of two hyperbolic periodic orbits, for all other $E\in[E_a,E_d]$ it consists exactly one hyperbolic periodic orbit.
\end{theo}

This theorem will be proved by showing the non-degeneracy of the minimal point of certain action function. Toward this goal, let us split the interval into suitably many subintervals $[E_a,E_d]=\cup_{i=0}^{k}[E_i-\delta_{E_i},E_i+\delta_{E_i}]$ with suitably small $\delta_{E_i}>0$. Once the open-dense property holds for each small subinterval, then it hold for the whole interval.

Let us explain how the interval $[E_a,E_d]$ is split. In the channel, one can choose a path along which the $\alpha$-function monotonely increases. Restricted on this path, we obtain a family of Lagrangians with one parameter. By using the method of \cite{BC}, we can see that it is typical that the minimal measure is supported at most on two periodic orbits for each class on this path. Thus, the Mather set $\tilde{\mathcal{M}}(E,g)$ consists of at most two periodic orbits for each $E\in[E_a,E_d]$.

Without of losing generality, we assume $g=(0,1)$, all these minimal curves are then associated with the homological class. Restricted on the neighborhood $\mathbb{S}_{\gamma_{E_i}}\subset \mathbb{T}^2$ of a minimal curve $\gamma_{E_i}\in\mathcal{M}(E_i,g)$ for certain energy $E_i$, we introduce a configuration coordinate transformation $x=X(u)$ such that along the curve $\gamma_{E_i}$ one has $u_1=\text{\rm constant}$. In the new coordinates, the Lagrangian reads
$$
L'(\dot u,u)=L(DX(u)\dot u,X(u))
$$
which is obviously positive definite in $\dot u$. As $\gamma_E(t)$ is a solution of the Euler-Lagrange equation determined by $L$, the curve $X^{-1}(\gamma_E)(t)$ solves the equation determined by $L'$ and is minimal for the action of $L'$. As there are at most two minimal curves for each energy, the neighborhood of these two curves can be chosen not to overlap each other. Thus one can extend the coordinate transformation to the whole torus.

Let $H'$ be the Hamiltonian determined by $L'$ through the Legendre transformation. the minimal curve determines a periodic solution for the Hamiltonian equation. By construction, $\partial_{v_2}H'>0$ holds along the periodic solution which entirely stays in the energy level set $H'^{-1}(E)$. We choose suitably small $\delta_{E_i}>0$ such that for $E\in[E_i-\delta_{E_i},E_i+\delta_{E_i}]$ each minimal periodic curve in $\mathcal{M}(E,g)$ falls into the strip $\mathbb{S}_{\gamma_{E_i}}$ and $\partial_{v_2}H'>0$ holds along each minimal periodic orbit.

For brevity of notation, we still use $x$ to denote the configuration coordinates, use $L$ and $H$ to denote the Lagrangian and Hamiltonian, for which the condition $\partial_{y_2}H>0$ holds along minimal periodic orbits for $E\in[E_i-\delta_{E_i},E_i+\delta_{E_i}]$. Under such conditions, the Lagrangian as well as the Hamiltonian can be reduced to a time-periodic system with one degree of freedom when it is restricted on energy level set. The new Hamiltonian $\bar H(x_1,y_1,\tau,E)$ solves the equation $H(x_1,y_1,x_2,\bar H)=E$ with $\tau=-x_2$, from which one obtains a new Lagrangian $\bar L=\dot x_1y_1-\bar H(x_1,y_1,\tau,E)$ where $y_1=y_1(x_1,\dot x_1,\tau)$ solves the equation $\dot x_1=\partial_{y_1}\bar H(x_1,y_1,\tau)$. In the following we omit the subscript ``1", i.e.  let $(x,y,\dot x)=(x_1,y_1,\dot x_1)$ if no danger of confusion.

We introduce a function of Lagrange action $F(\cdot,E)$: $\mathbb{T}\to\mathbb{R}$:
$$
F(x,E)=\inf_{\gamma(0)=\gamma(2\pi)=x} \int_{0}^{2\pi}\bar L(d\gamma(\tau),\tau,E)d\tau.
$$
A periodic curve $\gamma$ is called the minimizer of $F$ if the Lagrange action along this curve reaches the quantity $F(\gamma(0),E)$. As there might be two or more minimizers if $x$ is not a minimal point, the function $F$ may not be smooth in global. However, we claim that it is smooth in certain neighborhood of minimal point.

To verify our claim, we let $T_i=\frac {2\pi i}m$ and define the function of action $F_i(x,x',E)$
$$
F_i(x,x',E)=\inf_{\stackrel {\gamma(T_i)=x}{\scriptscriptstyle \gamma(T_{i+1})=x'}} \int_{T_i}^{T_{i+1}}\bar L(d\gamma(\tau),\tau)d\tau.
$$
There will be two or more minimizers of $F_i(x,x',E)$ if the point $x$ is in the ``cut locus" of the point $x'$. However, the minimizer is unique if $x$ is suitably close to $x'$, denoted by $\gamma_i(\cdot,x,x',E)$. In this case, it uniquely determines a speed $v=v(x,x')$ such that $\dot\gamma_i(T_i,x,x',E)=v(x,x')$. Let $\vec{x}= (x_0,x_1,\cdots,x_{m})$ denote a periodic configuration ($x_0=x_m$), we introduce a function of action
$$
\text{\bf F}(\vec{x},E)=\sum_{i=0}^{m-1}F_i(x_i,x_{i+1},E).
$$
As $T_{i+1}-T_i$ is suitably small and the Lagrangian is positive definite in the speed, the boundary condition $\gamma(T_j)=x_j$  for $j=i,i+1$ uniquely determines the speed $v_j=\dot\gamma(T_j)$ for $j=i, i+1$.  Indeed, the function $F_i$ generates an area-preserving twist map from the time-$T_i$-section to the time-$T_{i+1}$-section $\Phi_i$: $(x_i,y_i)\to(x_{i+1},y_{i+1})$
$$
y_{i+1}=\partial_{x_{i+1}}F_i(x_{i+1},x_i), \qquad y_{i}=-\partial_{x_{i}}F_i(x_{i+1},x_i).
$$
where $y_i=\partial_{\dot x}L(x_i,v_i,T_i)$. As the Lagrangian is positive definite in $\dot x$, it implies that the initial condition $(x_i,v_i)$ smoothly depends on the boundary condition $(x_i,x_{i+1})$ in this case. Because of the smooth dependance of solution of ordinary differential equation on initial condition, the function is smooth. Obviously, each minimal point of $\text{\bf F}(\cdot,E)$ uniquely determines a $c$-minimal measure with $c\in\alpha^{-1}(E)\cap \mathbb{C}_g$, supported on a periodic orbit $(\gamma_E,\dot\gamma_E)$ with $[\gamma_E]=g$. Let $x_i=\gamma_E(T_i)$, it satisfies the discrete Euler-Lagrange equation
$$
\frac{\partial F_i}{\partial x'}(x_{i-1},x_i,E)+\frac{\partial F_{i+1}}{\partial x}(x_{i},x_{i+1},E)=0.
$$

We claim that the periodic orbit is hyperbolic if and only if the minimal point (configuration) is non-degenerate, namely, the Jacobi matrix is positive definite:
$$
\text{\bf J}=\left[\begin{matrix}
A_0 & B_0 & 0& \cdots & B_{m-1}\\
B_0 & A_1 & B_1 & \cdots & 0 \\
0 & B_1 & A_2 & \cdots & 0 \\
\vdots & \vdots & \vdots & \ddots & B_{m-2}\\
B_{m-1} & 0 & 0 & B_{m-2} & A_{m-1}
\end{matrix}\right]
$$
where
$$
A_i=\frac{\partial^2F_{i-1}}{\partial x'^2}(x_{i-1},x_i)+\frac{\partial^2F_{i}}{\partial x^2}(x_i,x_{i+1}),\ \  B_i=\frac{\partial^2F_{i}}{\partial x\partial x'}(x_{i},x_{i+1})
$$
and $x_{-1}=x_{m-1}$.

Let $\vec{x}$ be a minimal configuration of the function $\text{\bf F}(\cdot,E)$, then the Jacobi matrix of $\text{\bf F}$ at the configuration $\vec{x}$ is non-negative. As the generating function $F_i(x,x',E)$ determines an area-preserving and twist map $\Phi_i$, we have $B_i<0$. Consequently, by using a theorem in \cite{vM}, we find that the smallest eigenvalue is simple. Let $\lambda_i$ denote the $i$-th eigenvalue of the matrix, at the minimal configuration one has
$$
0\le\lambda_0<\lambda_1\le\lambda_2<\cdots\le\lambda_{m-1}.
$$
Let $\xi_i=(\xi_{i,0},\xi_{i,1},\cdots,\xi_{i,m-1})$ be the eigenvector for $\lambda_i$. By choosing $\xi_{0,0}=1$ we have $\xi_{0,i}>0$ for $1\le i<m$ (see Lemma 3.4 in \cite{An}). At the minimal configuration, we find the following matrix is positive definite:
$$
\text{\bf J}_{m-1}=\left[\begin{matrix}
A_1 & B_1 & \cdots & 0 \\
B_1 & A_2 & \cdots & 0 \\
\vdots & \vdots & \ddots & B_{m-2}\\
0 & 0 & B_{m-2} & A_{m-1}
\end{matrix}\right].
$$
If not, there will be a vector $\hat v=(v_1,\cdots,v_{m-1})\in\mathbb{R}^{m-1}\backslash\{0\}$ such that $\hat v^t\text{\bf J}_m\hat v=0$. It follows that $v^t\text{\bf J}v=0$ if we set $v=(0,\hat v)\in\mathbb{R}^m$. As the matrix $\text{\bf J}$ is non-negative, it implies that $v=\mu\xi_0$, but it contradicts the fact that all entries of $\xi_{0}$ have the same sign, either positive or negative.

In a suitably small neighborhood $\vec{U}$ of the minimal configuration, let us consider the equations
\begin{equation}\label{variationeq1}
\frac{\partial\text{\bf F}}{\partial x_i}(x_0,x_1,\cdots,x_{m-1},E)=0,\qquad \forall\ i=1,2,\cdots,m-1.
\end{equation}
Since the matrix $\{\frac{\partial^2\text{\bf F}}{\partial x_i\partial x_j}\}_{i,j=1,\cdots,m-1}=\text{\bf J}_{m-1}$ is positive definite at the minimal point, by the implicit function theorem, the equation has a unique solution $x_i=X_i(x_0,E)$ for $x_0\in U_0=\pi\vec{U}$ which is smooth. Let $\gamma$: $[0,2\pi]\to\mathbb{R}$ be a minimizer of $F(x_0)$ with $\gamma(0)=\gamma(2\pi)=x_0$, we obtain a configuration $x_i=\gamma(2i\pi/m)$. Clearly, $\partial_{x_i}\text{\bf F}=0$ holds at this configuration for each $i\ge 1$. It implies the uniqueness of the minimizer of $F(x_0,E)$ for $x_0\in U_0$. The minimal point of $F$ uniquely determines a minimal configuration of $\text{\bf F}$, thus, the function $F$ is smooth in certain neighborhood of its minimal point.

\section{Non-degeneracy of minimizers}
\setcounter{equation}{0}

In a neighborhood of the minimal point, let us study what change the function of action undergoes when the Lagrangian is under a perturbation of potential $L\to L+P$, where $P$: $\mathbb{T}^2\to\mathbb{R}$ is a potential. Let $\bar L'$ denote the reduced Lagrangian of $L+P$ and let $G=-(\partial_{y_2}H)^{-1}$, one has
$$
\bar L'=\bar L+GP+O(\|P\|^2).
$$
We denote the minimizer of $F(x,E)$ by $\gamma(t,x,E)$ as the point $x$ uniquely determines the minimizer. Let $\gamma'(t,x,E)$ and $F'(x,E)$ be the quantities defined for $\bar L'$ as the quantities $\gamma(t,x,E)$ and $F(x,E)$ defined for $\bar L$. By the definition of minimizer, we have
\begin{align*}
F'(x,E)-F(x,E)&=\int_0^{2\pi}\bar L'(d\gamma'(\tau))d\tau -\int_0^{2\pi}\bar L(d\gamma(\tau))d\tau\\
&\ge \int_0^{2\pi}G(d\gamma'(\tau),\tau)P(\gamma'(\tau))d\tau+o(\|\gamma'-\gamma\|,\|P\|),
\end{align*}
and
\begin{align*}
F'(x,E)-F(x,E)&=\int_0^{2\pi}\bar L'(d\gamma'(\tau))d\tau-\int_0^{2\pi}\bar L(d\gamma(\tau))d\tau\\
&\le \int_0^{2\pi}G(d\gamma(\tau),\tau)P(\gamma(\tau))d\tau+o(\|\gamma'-\gamma\|,\|P\|).
\end{align*}
Since the distance between these two curves $\gamma$ and $\gamma'$ is due to the perturbation $P$, we finally obtain
\begin{equation}\label{nondegenerate1}
F'(x,E)=F(x,E)+\mathscr{K}_EP(x)+\mathscr{R}_EP(x)
\end{equation}
where $\mathscr{R}_EP=o(\|P\|)$ and
$$
\mathscr{K}_EP(x)=\int_0^{2\pi}G(d\gamma(\tau,x,E),\tau)P(\gamma(\tau,x,E))d\tau.
$$
The field of smooth curves $\{\gamma(\cdot,x,E)\}$ defines an operator $P\to\mathscr{K}_EP$, which maps functions defined on $\mathbb{T}^2$ into the function space defined on $\mathbb{T}$. Obviously, both $\mathscr{K}_EP$ and $\mathscr{R}_EP$ are smooth in $x\in U_0$ and in $E$.

Unless the point $x$ is a minimizer of $F(\cdot,E)$, the curve $\gamma(\cdot,x,E)$ may have corner at $\tau=0\mod 2\pi$.
\begin{lem}\label{hyperlem1}
There exist constants $\varepsilon,\theta>0$ such that if $F(x,E)-\min F(\cdot,E)<\varepsilon$ and if $\gamma:$ $[0,2\pi]\to\mathbb{R}$ is a  minimizer of $F(x,E)$, then
\begin{equation*}\label{hyperbolicorbit2}
\|\dot\gamma(0)-\dot\gamma(2\pi)\|<\theta\sqrt{F(x,E)-\min F(\cdot,E)}.
\end{equation*}
\end{lem}
\begin{proof}
Let us consider the derivative of $F(\cdot,E)$. As the Lagrangian is positive definite, some positive constants $m_L>0$ exist such that
$$
\frac{\partial ^2\bar L}{\partial\dot x^2}\ge m_L, \qquad \forall\ (x,\dot x)\in T\mathbb{T}^2.
$$
Since $\gamma(0,x,E)=\gamma(2\pi,x,E)=x$, one has $\partial_{x}\gamma(0)=\partial_{x}\gamma(2\pi)=1$ and
\begin{align*}
\Big|\frac{\partial F}{\partial x}\Big|&=\Big|\int_{0}^{2\pi}\Big (\frac{\partial\bar L}{\partial\dot x} (d\gamma(\tau),\tau)\frac{\partial\dot\gamma}{\partial x}+\frac{\partial\bar L}{\partial x} (d\gamma(\tau),\tau)\frac{\partial\gamma}{\partial x}\Big )d\tau\Big|\\
&=\Big|\frac{\partial\bar L}{\partial\dot x}(\dot\gamma(0),\gamma(0),0)-\frac{\partial\bar L}{\partial\dot x}(\dot\gamma(2\pi),\gamma(2\pi),2\pi)\Big|\\
&\ge m_L|\dot\gamma(0)-\dot\gamma(2\pi)|,
\end{align*}
where the second equality follows from that $\gamma$ solves the Euler-Lagrange equation. If $\frac{\partial F}{\partial x}>0$ and if the lemma does not hold, by choosing $x'-x=-\sqrt{\Delta}$ ($\Delta=F(x,E)-\min F(\cdot,E)$) we obtain from the Taylor series up to second order that
\begin{align*}
F(x',E)-\min F(\cdot,E)&=F(x',E)-F(x,E)+F(x,E)-\min F(\cdot,E)\\
&\le -\partial_{x}F(x,E)\sqrt{\Delta}+\frac M2\Delta+\Delta<0
\end{align*}
if $\theta>\frac 1{m_L}(1+\frac M2)$, where $M=\max\partial^2_{x}F$. But it is absurd. The case $\frac{\partial F}{\partial x}<0$ can be proved by choosing $x'-x=\sqrt{\Delta}$. This completes the proof.
\end{proof}

Let $x\in(x^*-\delta_{x^*},x^*+\delta_{x^*})$, where $x^*$ is the minimal point of $F(\cdot,E_0)$. As it was shown above, $\gamma(\frac{2i\pi}m,x,E_0)$ smoothly depends on $x$. Thus, we have a smooth foliation of curves in a neighborhood of the curve $\gamma(\cdot,x^*,E_0)$, the corner at $\gamma(0,x,E_0)$ approaches to zero as $F(x,E_0)\downarrow\min F(\cdot,E_0)$. For each $x$, if there is a corner at $\gamma(0,x,E_0)=\gamma(2\pi,x,E_0)$, we construct a curve $\gamma_{x}$ that smoothly connects the point $\gamma(2\pi-\delta,x,E_0)$ to the point $\gamma(\delta,x,E_0)$ with $\gamma_x(0)=x$, where $\delta>0$ is suitably small. Replacing the segment $\gamma(\cdot,x,E_0)|_{[0,\delta]\cup[2\pi-\delta,2\pi]}$ by this curve, we obtain a smooth curve $\gamma_{x}$ such that $\gamma_x(t)=\gamma(t,x,E_0)|_{[\delta,2\pi-\delta]}$ and $\gamma_x(0)=x$. Indeed, as the curve $\gamma(t,x,E_0)$ is $C^3$-smooth in $x$ and $\gamma(t,x^*,E_0)$ is also $C^3$-smooth in $t$, for small number $\varepsilon$ some $\mu_0>0$ exists such that the quantities
$$
\Big|\frac{d^k\gamma}{dt^k}(t,x,E_0)-\frac{d^k\gamma}{dt^k}(t,x^*,E_0))\Big|< \varepsilon,\qquad \forall\ x\in[x^*-\delta_{x^*},x^*+\delta_{x^*}], \ k=0,1,2,3.
$$
Let the curve $\zeta_x(\cdot)$: $[-\delta,\delta]\to\mathbb{R}$ be an interpolation polynomial of degree eight such that
$$
\frac{d^k\zeta_x}{dt^k}(t)=\frac{d^k\gamma}{dt^k}(t,x,E_0)-\frac{d^k\gamma}{dt^k}(t,x^*,E_0) \qquad \forall\ t=\pm\delta,
$$
and $\zeta_x(0)=\gamma(0,x,E_0)-\gamma(0,x^*,E_0)$, then the coefficients of the polynomial are smooth in $x$. Let $\gamma_{x}(t)=\gamma(t,x^*,E_0)+\zeta_x(t)$, we see that the foliation of the curves $\gamma_{x}$ is smooth in $x$ and as a function of $t$, $\gamma_{x}-\gamma(\cdot,x,E_0)$ is small in $C^3$-topology.

For each point $(\tau,x)\in\mathbb{S}$, there is a curve $\gamma_{x_0}$ such that $x=\gamma_{x_0}(\tau)$. It uniquely determines a speed $v=v(x,\tau)=\dot\gamma_{x_0}(\tau)$. By the construction, $v(x,\tau)$ is $C^3$-smooth in $(x,\tau)$. As $-G^{-1}=\partial_{y_2}H>0$ when it is restricted to a neighborhood of the minimal curves, both $v$ and $G$ can be approximated by $C^r$-function $v_s$ and $G_s$ in $C^3$-topology respectively i.e. $\|v-v_s\|_{C^3}<\varepsilon$ and $\|G-G_s\|_{C^3}<\varepsilon$ hold for small $\varepsilon>0$. Given a $C^r$-function $\bar P$: $\mathbb{T}\to\mathbb{R}$ we obtain a $C^r$-function $P=\mathscr{T}_{E_0}\bar P$: $\mathbb{T}^2\to\mathbb{R}$ defined by
\begin{equation}\label{nondegenerate2}
P(x,\tau)=\mathscr{T}_{E_0}\bar P(x_0)=G^{-1}_s(v_s(x,\tau),x,\tau)\bar P(x_0).
\end{equation}
By the definition, we have
\begin{equation}\label{nondegenerate3}
\mathscr{K}_E\mathscr{T}_{E_0}\bar P(x)= \int_0^{2\pi}\frac{G(d\gamma(\tau,x,E),\tau)}{G_s(v_s(\gamma (\tau,x,E),\tau),\gamma (\tau,x,E),\tau)}\bar P(x+\Delta\gamma (\tau,x,E))d\tau
\end{equation}
where $\Delta\gamma(\tau,x,E)$ is defined as follows: passing through the point $\gamma (\tau,x,E)$ there is a unique $x'$ such that $\gamma_{x'}(\tau)=\gamma(\tau,x,E)$. We set $\Delta\gamma(\tau,x,E)=x'-x$.

We introduce a set of perturbations with four parameters:
$$
\bar{\mathfrak{P}}=\Big\{\sum_{\ell=1}^2(A_{\ell}\cos \ell x+B_{\ell}\sin \ell x):\ (A_1,B_1,A_2,B_2)\in\mathbb{I}^4\Big\},
$$
where $\mathbb{I}^4=[1,2]\times[1,2]\times[1,2]\times[1,2]$. By applying the formula (\ref{nondegenerate3}) to the function $\cos\ell x$ and $\sin\ell x$  we find that
\begin{align}\label{nondegenerate4}
\mathscr{K}_E\mathscr{T}_{E_0}\cos\ell x&=u_{\ell}(x,E)\cos\ell x-v_{\ell}(x,E)\sin\ell x,\\
\mathscr{K}_E\mathscr{T}_{E_0}\sin\ell x&=u_{\ell}(x,E)\sin\ell x+v_{\ell}(x,E)\cos\ell x,\notag
\end{align}
where
\begin{align*}
u_{\ell}(x,E)&=\int_0^{2\pi}\frac{G(d\gamma(\tau,x,E),\tau)}{G_s(v_s(\gamma (\tau,x,E),\tau),\gamma (\tau,x,E),\tau)}\cos \ell\Delta\gamma(\tau,x,E)d\tau,\\
v_{\ell}(x,E)&=\int_0^{2\pi}\frac{G(d\gamma(\tau,x,E),\tau)}{G_s(v_s(\gamma (\tau,x,E),\tau),\gamma (\tau,x,E),\tau)}\sin \ell\Delta\gamma(\tau,x,E)d\tau.
\end{align*}
Let us study the dependence of the terms $u_{\ell}(x,E)$ and $v_{\ell}(x,E)$ on the point $x$. We claim that there exists constant $\theta_1>0$ as well as small numbers $\delta_{E_0}>0$ and $\delta_{x^*}>0$ such that for each $E\in (E_0-\delta_{E_0},E_0+\delta_{E_0})$, each $x\in(x^*-\delta_{x^*},x^*+\delta_{x^*})$, $j=0,1,2,3$ and $\ell=1,2$, we have
$$
|u_{\ell}(x,E)|\ge 1-\theta_1\delta, \qquad |v_{\ell}(x,E)|\le\theta_1\delta,
$$
\begin{equation}\label{nondegenerate5}
\max_{j=1,2,3}\Big\{\Big|\frac {\partial^ju_{\ell}}{\partial x^j}(x,E)\Big|,\Big|\frac {\partial^j v_{\ell}}{\partial x^j}(x,E)\Big|\Big\} \le\theta_1\delta.
\end{equation}

By the construction of the curves $\gamma_x$, for $\tau\in\mathbb{T}\backslash (-\delta,\delta)$ and for $x\in(x^*-\delta_{x^*},x^*+\delta_{x^*})$ we have
$$
\Delta\gamma(\tau,x,E_0)=0 \ \ \text{\rm and}\ \ \frac{G(d\gamma(\tau,x,E_0),\tau)}{G(v(\gamma (\tau,x,E_0),\tau),\gamma (\tau,x,E_0),\tau)}=1
$$
and $\partial_{x}^j\Delta\gamma (\tau,x,E_0)$ is small for $\tau\in(-\delta,\delta)$ and for $j=0,1,2,3$. Integrating the them over the set with Lebesgue measure $2\delta$, we find that some small $\theta_1>0$ exists such that the formulae in (\ref{nondegenerate5}) hold for $E=E_0$ with $\theta_1$ being replaced by $\theta_1/4$ if $G_s$ and $v_s$ in the formula (\ref{nondegenerate3}) are replaced by $G$ and $v$ respectively. As both $v$ and $G$ are approximated by $v_s$ and $G_s$ in $C^3$-topology, by choosing $\varepsilon>0$ suitably small, all formulae in (\ref{nondegenerate5}) hold for $E=E_0$ with $\theta_1$ being replaced by $\theta_1/2$.

For other energy $E$, let us recall the solution $x_i=X_i(x_0,E)$ of Eq. (\ref{variationeq1}) is smooth. As the map $\Phi_i$: $(x_i,y_i)\to(x_{i+1},y_{i+1})$ is area-preserving and twist, it uniquely determines the initial speed $v_0=v_0(x_0,E)$, namely, the initial speed smoothly depends on the initial position and such dependence is also smooth in the parameter $E$. As solution of ODE smoothly depends on its initial conditions, the minimal curve $\gamma(\cdot,x,E)$ of $F(x,E)$ smoothly depends on the parameters $x$ and $E$. Thus, the formulae in (\ref{nondegenerate5}) hold if the numbers $\delta_{E_0}>0$ and $\delta_{x^*}>0$ are suitably small.

\begin{theo}\label{theo2}
There exists an open-dense set $\mathfrak{O}\subset C^r(M,\mathbb{R})$ with $r\ge4$  such that for each $P\in\mathfrak{O}$ and each $E\in [E_0-\delta_{E_0},E_0+\delta_{E_0}]$, all minimizers of $F(\cdot,E)$, determined by $L+P$, are non-degenerate.
\end{theo}
\begin{proof}
To show the non-degeneracy of the global minimum of $F(\cdot,E)$ located at the point $x$, we only need to verify that
\begin{equation}\label{nondegenerate6}
F(x+\Delta x,E)-F(x,E)\ge M|\Delta x|^4
\end{equation}
holds for small $|\Delta x|$, where $M=12^{-1}\max\partial^4_{x}F$. Assume $I$ is an interval, we define $\text{\rm Osc}_{I}F=\max_{x,x'}|F(x)-F(x')|$. To show the non-degeneracy, it is sufficient to verify that
$$
\text{\rm Osc}_{I}F(\cdot,E)\ge M|I|^4
$$
if the minimal point $x\in I$, where $|I|$ denotes the length of the interval.

The openness is obvious of $\mathfrak{P}$. To show the density, we are concerned only about the configurations where $F$ takes the value close to the minimum and consider small perturbations from the following set where the parameters $(A_1,B_1,A_2,B_2)$ range over the cube $\mathbb{I}^4=[1,2]\times[1,2]\times[1,2]\times[1,2]$
$$
\mathfrak{V}_E=\Big\{(\mathscr{K}_E+\mathscr{R}_E)\mathscr{T}_{E_0}\sum_{\ell=1}^2\epsilon (A_{\ell}\cos\ell x +B_{\ell}\sin\ell x):(A_1,B_1,A_2,B_2)\in\mathbb{I}^4\Big\}
$$
where each element is a function of $(x,E)$, see the formulae (\ref{nondegenerate4}). Recall that both  operators $\mathscr{K}_E$ and $\mathscr{T}_{E_0}$ are linear and  $\|\mathscr{R}_E(\epsilon P)\|=o(\epsilon)$, see the formula (\ref{nondegenerate1}) and the formula (\ref{nondegenerate2}).

We choose sufficiently large integer $K$ so that $\epsilon=\sqrt[4]{\pi/K}$ can be arbitrarily small. Let $x_k=\frac{2k\pi}K$, $I_k=[x_k-d,x_k+d]$ and $d=\pi/K$, then $\bigcup_{k=0}^{K-1}I_k=\mathbb{T}$. Restricted on each interval $I_k$, each $C^4$-function $V\in\mathfrak{V}_E$ is approximated by the Taylor series (module constant)
$$
V_k(x)=\epsilon\Big(a_k(x-x_{k})+b_k(x-x_{k})^2+c_k(x-x_{k})^3+O(|x-x_{k}|^4)\Big).
$$

Given two points $(a_k,b_k,c_k)$ and $(a'_k,b'_k,c'_k)$, we have two functions $V_k$ and $V'_k$ in the form of Taylor series. Let $\Delta V=V'_k-V_k$, $\Delta a=a'_k-a_k$, $\Delta b=b'_k-b_k$ and $\Delta c=c'_k-c_k$, we have $\Delta V(x_k)=0$ and
\begin{align*}
&\Delta V(x_k+d)+\Delta V(x_k-d)=2\epsilon\Delta bd^2+O(\epsilon d^4),\\
&\Delta V(x_k+d)-\Delta V(x_k-d)=2\epsilon(\Delta a+\Delta cd^2)d+O(\epsilon d^4),\\
&\Delta V\Big(x_k\pm\frac 12d\Big)=\epsilon\Big(\pm\frac 12\Delta a+\frac 14\Delta bd\pm\frac 18\Delta cd^2\Big)d+O(\epsilon d^4).
\end{align*}
It follows that
\begin{equation}\label{nondegenerate7}
\text{\rm Osc}_{I_k}(V'_k-V_k)\ge\frac {\epsilon}4\max\Big\{|\Delta a|,|\Delta b|d,|\Delta c|d^2\Big\} d.
\end{equation}

We construct a grid for the parameters $(a_k,b_k,c_k)$ by splitting the domain for them equally into a family of cuboids and setting the size length by
$$
\Delta a_k=12Md^{\frac{11}4},\ \ \Delta b_k=12Md^{\frac{7}4},\ \ \Delta c_k=12Md^{\frac 34}.
$$
These cuboids are denoted by $\text{\uj c}_{kj}$ with $j\in\mathbb{J}_k=\{1,2,\cdots\}$, the cardinality of the set of the subscripts is up to the order
$$
\#(\mathbb{J}_k)=N[d^{-\frac{21}4}],
$$
where the integer $0<N\in\mathbb{N}$ is independent of $d$. If $\text{\rm Osc}_{I_k}F(\cdot,E)\le Md^4$, we obtain from the formula (\ref{nondegenerate7}) that
$$
\text{\rm Osc}_{I_k}(F(x,E)+V(x))\ge 2Md^4
$$
if $V(x)=\epsilon(a(x-x_k)+ b(x-x_k)^2+c(x-x_k)^3+O(|x-x_k|^4))$ with
$$
\max\Big\{|a|d^{-\frac{11}4},|b|d^{-\frac 74}, |c|d^{-\frac 34}\Big\}\ge 12M.
$$

The coefficients $(a_k,b_k,c_k)$ depend on the parameters $(A_1,B_1,A_2,B_2)$, the energy $E$ and the position $x_k$. The gird for $(a_k,b_k,c_k)$ induces a grid for the parameters $(A_1,B_1,A_2,B_2)$, determined by the equation
\begin{equation}\label{nondegenerate8}
\left[\begin{matrix} a_k\\  b_k\\  c_k
\end{matrix}\right]=({\bf C_1}{\bf U}+{\bf C_2})
\left[\begin{matrix} A_1\\  B_1\\  A_2\\  B_2
\end{matrix}\right]\Big(1+T_{\epsilon,E,x_{k}}(A_1,B_1,A_2,B_2)\Big)
\end{equation}
where the map $T_{\epsilon,E,x_{k}}$: $\mathbb{R}^4\to\mathbb{R}^3$ is as small as of order $O(\epsilon)$,
$$
{\bf C_1}=\left[\begin{matrix}
-\sin x_{k} &  \cos x_{k}  & -2\sin 2x_{k} &  2\cos 2x_{k}\\
-\cos x_{k} & -\sin x_{k}  & -4\cos 2x_{k} & -4\sin 2x_{k} \\
 \sin x_{k} & -\cos x_{k}  &  8\sin 2x_{k} & -8\cos 2x_{k}
\end{matrix}\right],
$$
$$
{\bf U}=\text{\rm diag}\left\{
\left[\begin{matrix}
u_1(x_{k}) & v_1(x_{k})\\
-v_1(x_{k}) & u_1(x_{k})
\end{matrix}\right],
\left[\begin{matrix}
 u_2(x_{k}) & v_2(x_{k})\\
-v_2(x_{k}) & u_2(x_{k})
\end{matrix}\right]\right\},
$$
each entry of ${\bf C_2}$ is a linear function of $\partial^{j}_{x}u_{\ell}\cos \ell x_{k}$, $\partial^{j}_{x}v_{\ell}\cos\ell x_{k}$, $\partial^{j}_{x}u_{\ell}\sin\ell x_{k}$ and $\partial^{j}_{x}v_{\ell}\sin\ell x_{k}$ with $j=1,2,3$, $\ell=1,2$. Both matrices ${\bf U}$ and ${\bf C}_2$ depend on the energy $E$, ${\bf U}$ is close to the identity matrix.  Let ${\bf M_1}$ be the matrix composed by the first three columns of ${\bf C_1}{\bf U}+{\bf C_2}$, ${\bf M_2}$ be the matrix composed by the first, the second and the fourth column of ${\bf C_1}{\bf U}+{\bf C_2}$. As we are concerned about those positions where $F$ takes value close to the minimum and about the energy $E$ close to $E_0$, in virtue of (\ref{nondegenerate5}) we obtain
\begin{align*}
\text{\rm det}({\bf M}_1)(x_k)&=6\sin 2x_{k}(1-O(\theta_1\delta)), \\
\text{\rm det}({\bf M}_2)(x_k)&=-6\cos 2x_{k}(1-O(\theta_1\delta)).
\end{align*}
As $\inf_{x_{k}}\{|{\rm det}{\bf M}_1(x_k)|,|{\rm det}{\bf M}_2(x_k)|\}= 3\sqrt{2}(1-O(\theta_1\delta))$, the grid for $(a_k,b_k,c_k)$ induces a grid for $(A_1,B_1,A_2,B_2)$ which contains as many as $N_1[d^{-\frac{21}4}]$ 4-dimensional strips ($N_1>0$ is independent of $d$). Note that the induced partition for $(A_1,B_1,A_2,B_2)$ depends on the energy.

Given an energy $E\in[E_0-\delta_{E_0},E_0+\delta_{E_0}]$, if there exist Taylor coefficients $(a_k,b_k,c_k)$ which determines a perturbation $V$ such that
$$
\text{\rm Osc}_{I_k}(F(\cdot,E)+V)\le Md^{4}
$$
then for $(a'_k,b'_k,c'_k)$ which determines a perturbation $\Delta V'$  one obtains from the formula (\ref{nondegenerate7}) that
\begin{equation}\label{nondegenerate9}
\text{\rm Osc}_{I_k}(F(\cdot,E)+V')\ge 2Md^{4}
\end{equation}
provided
$$
\max\Big\{\frac{|a_k-a'_k|}{12Md^{\frac{11}4}},\frac{|b_k-b'_k|}{12Md^{\frac 74}}, \frac{|c_k-c'_k|}{12Md^{\frac 34}}\Big\}\ge 1.
$$

Under the map defined by the formula (\ref{nondegenerate8}), the inverse image of a cuboid $\text{\uj c}_{k}$ with the size $24Md^{\frac{11}4}\times24Md^{\frac{7}4}\times24Md^{\frac{3}4}$ is a strip in the parameter space of $(A_1,B_1,A_2,B_2)$, denoted by $\text{\uj S}_{k}(E)$, with the Lebesgue measure as small as $N_1^{-1}d^{\frac{21}4}$. If the cuboid $\text{\uj c}_{k}$ is centered at $(a_k,b_k,c_k)$, then for $(a'_k,b'_k,c'_k)\notin\text{\uj c}_{k}$ the inequality (\ref{nondegenerate9}) holds.

Splitting the interval $[E_0-\delta_{E_0},E_0+\delta_{E_0}]$ equally into small sub-intervals $I_{E,j}$ with the size $|I_{E,j}|=M_1^{-1}d^4$, we obtain as many as $[M_1d^{-4}]$ small intervals. As the function $F$ is Lipschitz in $E$, suitably large positive number $M_1$ can be chosen so that
$$
\max_{x\in I_k}|F(x,E)-F(x,E')|<\frac 12Md^4, \qquad \forall\ E,E'\in I_{E,j}.
$$
Therefore, for $V\in\mathfrak{V}_E$ with $(\Delta A_1,\Delta B_1,\Delta A_2,\Delta B_2)\notin\text{\uj S}_{k}(E)$, one has
\begin{equation}\label{nondegenerate10}
\text{\rm Osc}_{I_k}(F(\cdot,E)+\Delta V')\ge Md^{4}.
\end{equation}
Pick up one energy $E_{j}$ in each small interval $I_{E,j}$, there are $[M_1d^{-4}]$ strips $\text{\uj S}_{k}(E_j)$. Finally, by considering all small intervals $I_k$ with $k=0,1,\cdots K-1$, we find
$$
\text{\rm meas}\Big(\bigcup_{k,j}\text{\uj S}_{k}(E_j)\Big)\le M_1N_1^{-1}\sqrt[4]{d}.
$$
Let $\text{\uj S}^c=\mathbb{I}^4\backslash\cup_{j,k}\text{\uj S}_{kj}(E_j)$, we obtain the Lebesgue measure estimate
$$
\text{\rm meas}(\text{\uj S}^c)\ge 1-M_1N_1^{-1}\sqrt[4]{d}\to 1,\qquad \text{\rm as}\ \ d\to 0.
$$
Obviously, for any $(A_1,B_1,A_2,B_2)\in\text{\uj S}^c$, any $E\in [E_0-\delta_{E_0},E_0+\delta_{E_0}]$ and any $k=1,2,\cdots,K$ the formula (\ref{nondegenerate10}) holds. This proves that it is open-dense that all minimal points of $F(\cdot,E)$ are non-degenerate when the energy ranges over the interval $[E_0-\delta_{E_0},E_0+\delta_{E_0}]$.
\end{proof}

\section{Hyperbolicity}
\setcounter{equation}{0}

Let $x^*$ be a minimal point of the function $F(\cdot,E)$ and let the curve $\gamma(\cdot,x^*,E)$: $\mathbb{T}\to\mathbb{R}$ be the minimizer of $F(x^*,E)$ which is smooth and determines a periodic orbit $(\tau,\gamma(\tau),\frac d{d\tau}\gamma(\tau))$ of the Lagrange flow $\phi^{\tau}_{\bar L}$. Back to the autonomous system, it determines a periodic orbit $(\gamma_1(t),\dot\gamma_1(t),\gamma_2(t),\dot\gamma_2(t))$ of the Lagrange flow $\phi_L^t$, where $\gamma_2(t)=-\tau$, $\gamma_1(t)=\gamma(\gamma_2(t))$.
\begin{theo}\label{theo3}
If $x^*$ is a non-degenerate minimal point of the function $F(\cdot,E)$, then the periodic orbit $\gamma(\cdot,x^*,E)$ is hyperbolic.
\end{theo}
\begin{proof}
If a periodic orbit is hyperbolic, it has its stable and unstable manifold in the phase space. Consequently, any orbit staying on the stable (unstable) manifold approaches to the periodic orbit exponentially fast as the time approaches to positive (negative) infinity.

In a small neighborhood of the minimal periodic curve $\gamma$, each point $x$ on the section $\{\tau=0\}$ determines at least one forward (backward) semi-static curve $\gamma^+_x$: $\mathbb{R}_+\to\mathbb{T}$ ($\gamma^-_x$: $\mathbb{R}_-\to\mathbb{T}$) such that $\gamma^{\pm}_x(0)=x$. These curves determine forward (backward) semi-static orbits $d\gamma^{\pm}_x$ of which the $\omega$-set ($\alpha$-set) is the periodic orbit $d\gamma$. In the configuration space $(x,\tau)\in\mathbb{T}^2$, these two curves intersect with the section $\{\tau=0\}$ infinitely many times at the points $\gamma^+_x(2k\pi)$ and $\gamma^-_x(-2k\pi)$. These points are denoted by $x_i$, they are well ordered $\cdots\prec x_{i+1}\prec x_i\cdots\prec x_0$. It is possible that $\gamma^+_x(2k\pi)=\gamma^-_x(-2k'\pi)$. In this case, we count the point twice. For each point $x_i$, there is a curve joining $(x_i,0)$ to $(x_i,2\pi)$ which is composed by some segments of $\gamma^+_x$ as well as of $\gamma^-_x$. For instance, in the following figure, by starting from the point $(x_2,0)$ and  following a segment of $\gamma^-_x$ to the point $A$, then following a segment of $\gamma^+_x$ to the point $B$ and finally following a segment of $\gamma^-_x$ to the point $(x_2,2\pi)$, we obtain a circle. Clearly, the Lagrange action along this circle is not smaller than the quantity $F(x_2)$.
\begin{figure}[htp] 
  \centering
  \includegraphics[width=5.2cm,height=5.0cm]{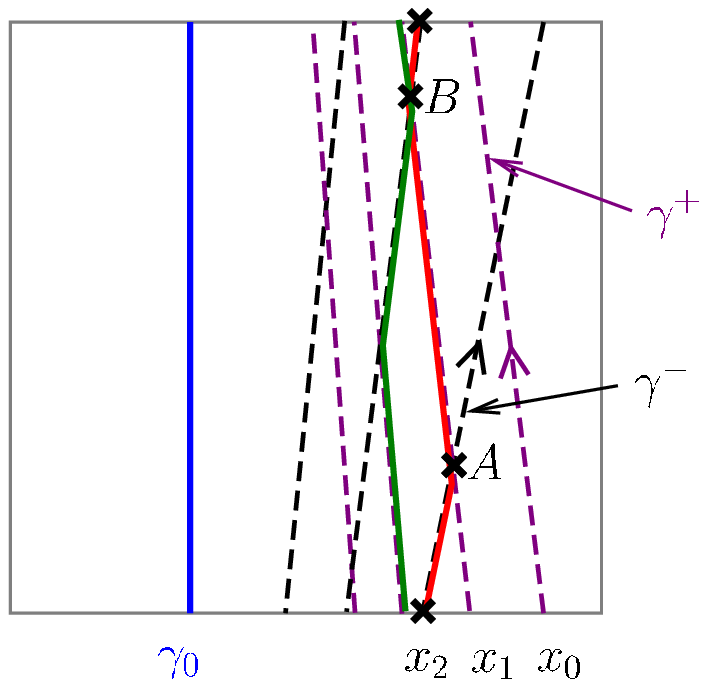}
  \label{}
\end{figure}

Let us consider the whole sequence $\{x_i\}$, we obtain infinitely many circles in that way. Therefore, the sum of the quantities $F(x_i)|_{i=0}^{\infty}$ is obviously not bigger than the total action along all of these circles
\begin{equation}\label{hypereq1}
\sum_{i=0}^{\infty}F(x_i)\le \lim_{k\to\infty}\Big\{\int_{0}^{2k\pi}L(d\gamma^+(\tau),\tau)d\tau+\int_{-2k\pi}^0L(d\gamma^-(\tau), \tau)d\tau\Big\}.
\end{equation}
The right hand side is nothing else but the barrier function valued at $x_0$.

As the periodic orbit supports the minimal measure, both $\gamma^+_x(2k\pi)$ and $\gamma^-_x(-2k\pi)$ approach the point $x^*$ where the periodic curve intersects the section $\{\tau=0\}$ as $k\to\infty$. If the periodic orbit is not hyperbolic, the sequence of $\{x_i\}$ approach $x$ slower than exponentially, i.e., for any small $\lambda>0$ there exists $\delta>0$ such that
\begin{align*}
|\gamma^+_x(2(k+1)\pi)-x^*)|&\ge (1-\lambda)|\gamma^+_x(2k\pi)-x^*)|,\\
|\gamma^-_x(-2(k+1)\pi)-x^*)|&\ge (1-\lambda)|\gamma^-_x(-2k\pi)-x^*)|,
\end{align*}
if $|\gamma^{\pm}_x(0)-x^*|\le\delta$. It follows that $|x_{i+1}-x^*|\ge(1-\lambda)|x_{i}-x^*|$. As the periodic curve is assumed non-degenerate minimizer, some $\lambda_0>0$ exists such that
\begin{equation}\label{hypereq2}
\sum_{i=0}^{\infty}(F(x_i)-F(x^*))\ge\lambda_0\sum_{i=0}^{\infty}(x_i-x^*)^2\ge \lambda_0\frac{(x_0-x^*)^2}{1-(1-\lambda)^2}.
\end{equation}
By subtracting $\min F$ from the Lagrangian $L$ we obtain that
$$
\text{\rm right-hand-side of (\ref{hypereq1})}=u^-(x,0)-u^+(x,0)
$$
where $u^{\pm}$ represents the backward (forward) weak-KAM solution. Since $u^-$ is semi-concave and $u^+$ is semi-convex, $u^--u^+$ is semi-concave. Since $(x^*,0)$ is a minimal point where $u^-(x^*,0)-u^+(x^*,0)=0$, there exists some number $C_L>0$ such that (cf. \cite{Fa})
$$
u^-(x_0,0)-u^+(x_0,0)\le C_L(x_0-x^*)^2.
$$
Comparing this with the inequality (\ref{hypereq2}), we obtain from (\ref{hypereq1}) a contradiction
$$
\lambda_0\frac{(x_0-x^*)^2}{1-(1-\lambda)^2}\le C_L(x_0-x^*)^2
$$
if $\lambda>0$ is suitably small. This proves the hyperbolicity of the periodic orbit.
\end{proof}

We are now ready to prove the main result.

\noindent{\it Proof of Theorem \ref{theo1}}. According to Theorem \ref{theo2} and \ref{theo3}, for each $E_i\in[E_a,E_d]$, a neighborhood $[E_i-\delta_{E_i},E_i+\delta_{E_i}]$ of $E_i$ and an open-dense set $\mathfrak{O}(E_i)\subset C^r(\mathbb{T}^2,\mathbb{R})$ exist such that for each $P\in\mathfrak{O}(E_i)$ and each $E\in[E_i-\delta_{E_i},E_i+\delta_{E_i}]$ each minimal orbit of $\phi_{L+P}^t$ with homological class $g$ is hyperbolic. As each $\delta_{E_i}$ is positive, there exists finitely many $E_i$ such that $[E_a,E_d]\subset\cup_i[E_i-\delta_{E_i},E_i+\delta_{E_i}]$. We take $P\in\cap\mathfrak{O}(E_i)$, the hyperbolicity for $L+P$ holds for all $E\in[E_a,E_d]$.

Once a minimal point is non-degenerate for certain $E$, by the theorem of implicit function  it has natural continuation to a neighborhood of $E$. Namely, there exists a curve of minimal points passing through this point, it either reaches to the boundary of $[E_a,E_d]$, or extends to some point $E'$ where the critical point is degenerate. Since each global minimal point is non-degenerate, the critical point becomes local minimum when it enters into certain neighborhood of $E'$. As each non-degenerate minimal point is isolated to other minimal points for the same energy $E$, there are finitely many such curves, denoted by $\Gamma_i$.

For a curve $\Gamma_i$: $I_i=(E_i,E'_i)\to\mathbb{T}$, the definition domain $I_i$ contains finitely many closed sub-intervals $I_{i,j}$ such that $F(\Gamma_i(E),E)=\min_xF(\cdot,E)$ for all $E\in I_{i,j}$. By definition, $I_{i,j}\cap I_{i,j'}=\varnothing$ for $j\ne j'$. Let $\Gamma_{i,j}=\Gamma_i|_{I_{i,j}}$, we have finitely many curves $\{\Gamma_{i,j}\}$ such that $F(\cdot, E)$ reaches global minimum at the point $x$ if and only if $x=\Gamma_{i,j}(E)$ for certain subscript $(i,j)$.

For each $E\in\partial I_{i,j}$, by the definition of $I_{i,j}$, some other subscript $(i',j')$ exists such that $F(\cdot,E)$ reaches the global minimum at the points $\Gamma_{i,j}(E)$ and $\Gamma_{i',j'}(E)$. It is obviously an open-dense property that
$$
\frac {dF(\Gamma_{i,j}(E),E)}{dE}\ne \frac {dF(\Gamma_{i',j'})(E),E)}{dE}.
$$
Thus, it is also open-dense that $[E_a,E_d]=\cup I_{i,j}$ and $[E_a,E_d]\backslash\cup\text{\rm int}I_{i,j}$ contains finitely points. This completes the whole proof.\eop

\noindent{\it Proof of Theorem \ref{mainth}}. According to Theorem \ref{theo1}, for each class $g\in H_1(\mathbb{T}^2,\mathbb{Z})$ and each $[E_a,E_d]\subset(0,\infty)$,  an open-dense set $\mathfrak{V}(g,E_a,E_d)\subset C^r(\mathbb{T}^2,\mathbb{R})$ exists such that for each $V\in\mathfrak{V}(g,E_a,E_d)$ the minimal periodic orbits of $\phi^t_{L+V}$ are hyperbolic if they stay in the Mather set $\tilde{\mathcal{M}}(E,g)$ with $E\in[E_a,E_d]$.

Let $E_a^i\downarrow 0$, $E_d^i\uparrow\infty$ be sequences of numbers. Because there are countably many homological classes in the group $H_1(\mathbb{T}^2,\mathbb{R})$, the set
$$
\mathfrak{P}=\bigcap_{\stackrel {g\in H_1(\mathbb{T}^2,\mathbb{R})}{\scriptscriptstyle i\in\mathbb{N}}}\mathfrak{V}(g,E^i_a,E^i_d)
$$
is obviously residual in $C^r(\mathbb{T}^2,\mathbb{R})$. Therefore, to complete the proof we only need to consider periodic orbits lying in the Mather set $\tilde{\mathcal{M}}(c)$ with $c\in\mathbb{F}_0=\{c\in H^1(\mathbb{T}^2,\mathbb{R}):\alpha(c)=\min\alpha\}$.

The flat $\mathbb{F}_0$ is either two dimensional disk or a line. If it is a disk, it is a typical phenomenon that the minimal measure is uniquely supported on a fixed point or a shrinkable periodic orbit for each $c\in\text{\rm int}\mathbb{F}_0$ which is guaranteed by a result of Ma\~n\'e \cite{Man} as well as of Massart \cite{Mas}. It is also typical that this fixed point or a shrinkable closed orbit is hyperbolic. Let us consider $c\in\partial\mathbb{F}_0$ such that $\tilde{\mathcal{M}}(c)$ contains periodic orbits and assume that $g$ is the homological class of the periodic orbit. As it was studied in \cite{Ch}, certain $\lambda_0>0$ exists such that the channel $\mathbb{C}_g=\cup_{\lambda\ge\lambda_0>0}\mathscr{L}_{\beta}(\lambda g)$ is connected to the flat $\alpha^{-1}(\min\alpha)$. As the speed of the periodic orbit is not zero, one can see that the proof of Theorem \ref{theo1} applies to $[\min\alpha,E_d]$. If the flat $\mathbb{F}_0$ is a line, it is obviously typical that, for each $c\in\mathbb{F}_0$, the $c$-minimal measure is supported on two periodic orbits with different rotation vector, each of them is hyperbolic.\eop

\noindent{\bf Acknowledgement} This work is supported by National Basic Research Program of China (973 Program, 2013CB834100), NNSF of China (Grant 11201222, Grant 11171146),  Basic Research Program of Jiangsu Province (BK2008013) and a program PAPD of Jiangsu Province, China.


\begin{thebibliography}{WW}

\bibitem[An]{An} Angenent S., {\it The periodic orbits of an area preserving twist map}, Commun. Math. Phys. {\bf 115} (1988) 353-374.

\bibitem[BC]{BC} Bernard P. \& Contreras G., {\it A generic property of families of Lagrangian systems}, Annals of Math. {\bf 167} (2008) 1099-1108.

\bibitem[Ch]{Ch} Cheng C.-Q., {\it Arnold diffusion in nearly integrable Hamiltonian systems}, (2012) Preprint.

\bibitem[Fa]{Fa} Fathi A., {\it Weak KAM Theorem in Lagrangian Dynamics}, Cambridge Studies in Adavnced Mathematics, Cambridge University Press, (2009).

\bibitem[Mat]{Mat} Mather J., {\it Action minimizing invariant measures for positive definite Lagrangian systems,} Math. Z., {\bf 207}(1991) 169-207.

\bibitem[Man]{Man} Ma\~n\'e R., {\it Generic properties and problems of minimizing measures of Lagrangian systems}, Nonlinearity {\bf 9} (1996) 273-310.

\bibitem[Mas]{Mas} Massart D., {\it On Aubry sets and Mather's action functional}, Israel J. Math. {\bf 134} (2003) 157-171.

\bibitem[vM]{vM} van Moerbeke P., {\it The spectrum of Jacobi matrices}, Invent. Math. {\bf 37} (1976) 45-81.

\end{thebibliography}
\end{document}